\documentclass[12pt,reqno]{amsart}
\pdfoutput=1
\usepackage[headings]{fullpage}
\usepackage{amssymb,amsmath,bbm,tikz}
\usepackage{graphicx,color,enumerate, float}
\usepackage[all,cmtip]{xy}
\usepackage{url}
\usepackage[margin=1.2in]{geometry}
\usepackage{verbatim}
\usepackage{pb-diagram}

\usepackage[bookmarks=true,%
    colorlinks=true,%
    linkcolor=blue,%
    citecolor=blue,%
    filecolor=blue,%
    menucolor=blue,%
    urlcolor=blue,%
    breaklinks=true]{hyperref}

\newtheorem{theorem}{Theorem}[section]
\theoremstyle{definition}
\newtheorem{proposition}[theorem]{Proposition}
\newtheorem{lemma}[theorem]{Lemma}

\newtheorem{corollary}[theorem]{Corollary}
\newtheorem{conjecture}[theorem]{Conjecture}

\renewcommand{\setminus}{{\smallsetminus}}

\newcommand{\RR}{{\mathbb{R}}}
\newcommand{\Z}{{\mathbb{Z}}}

\newcommand{\QQ}{{\mathbb{Q}}}

\def\BN{\mathbbm N}

\def\be{\begin{equation}}
\def\ee{\end{equation}}

\graphicspath{{figures/}}

\begin{document}

% Octahedral decompositions of knots and the AJ conjecture
% Certificates and the AJ conjecture
\title[Hyperbolic $3$-manifolds with finite dimensional skein modules]{
Infinite families of hyperbolic $3$-manifolds with finite dimensional skein modules} 

\author{Renaud Detcherry}
\address{Max Planck Institute for Mathematics \\
         Vivatsgasse 7, 53111 Bonn, Germany \newline
         {\tt \url{http://people.mpim-bonn.mpg.de/detcherry}}
         }
\email{detcherry@mpim-bonn.mpg.de}

\thanks{
%The author was supported in part by a National Science Foundation
%grant DMS-0805078. \\
%\newline
1991 {\em Mathematics Classification.} Primary 57N10. Secondary 57M25.
}

\date{Friday 1 March, 2019}%\today}

\begin{abstract} The Kauffman bracket skein module $K(M)$ of a $3$-manifold $M$ is the quotient of the $\QQ(A)$-vector space spanned by isotopy classes of links in $M$ by the Kauffman relations. A conjecture of Witten states that if $M$ is closed then $K(M)$ is finite dimensional. We introduce a version of this conjecture for manifolds with boundary and prove a stability property for generic Dehn-filling of knots. As a result we provide the first hyperbolic examples of the conjecture, proving that almost all Dehn-fillings of any two-bridge knot satisfy the conjecture.
\end{abstract}

\maketitle

%%%%%%%%%%%%%%%%%%%%%%%%%%%%%%%%%%%%%%%%%%%%%%%%%%%%%%%%%%%%%%%%%%%%%%%%%%
%%%%%%%%%%%%%%%%%%%%%%%%%%%%%%%%%%%%%%%%%%%%%%%%%%%%%%%%%%%%%%%%%%%%%%%%%%

\section{Introduction}
\label{sec.intro}
The Kauffman bracket skein modules were introduced independent by Turaev \cite{Tu:skeinmodule} and Przytycki \cite{Prz91} as a way of generalizing the Jones polynomial (in its Kauffman bracket formulation \cite{Kau87}) to links in an arbitrary $3$-manifold. Although the concept of skein modules can be generalized to other contexts than just the Kauffman bracket skein module, in this paper, we will only consider Kauffman bracket skein modules and refer to them as just skein modules. Let $\mathcal{R}$ be a ring containing $\Z[A,A^{-1}].$ For $M$ a compact oriented $3$-manifold, with or without boundary, the skein module with $\mathcal{R}$-coefficients $K(M,\mathcal{R})$ is the $\mathcal{R}$-module defined by:
$$K(M,\mathcal{R})=\mathrm{Span}_{\mathcal{R}}(\ \textrm{banded links in M} \ )/_{\textrm{isotopy, Kauffman relations}}$$
where the Kauffman relations K1 and K2 are given in Figure \ref{fig:kauffman}.
\begin{figure}[!htpb]
\centering
\def \svgwidth{.6\columnwidth}
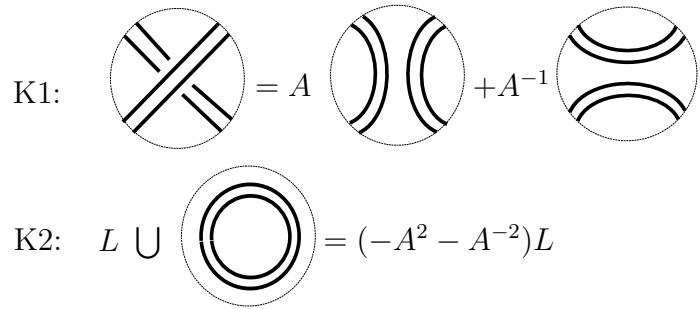
\caption{The Kauffman relations K1 and K2. K1 relates three links that differ in a small ball as shown, K2 simplifies split union of a link and a trivial knot in a small ball.}
\label{fig:kauffman}
\end{figure}
\\ Although the traditional definition simply uses $\mathcal{R}=\Z[A,A^{-1}],$ in this paper we will mostly consider coefficients $\mathcal{R}=\QQ(A).$ This has the effect of killing all the torsion in $K(M,\Z[A,A^{-1}])$ and moreover, $K(M,\QQ(A))$ is then a $\QQ(A)$-vector space.
\\ The skein module $K(M,\Z[A,A^{-1}])$ is often not finitely generated; in fact, by results of Bullock and Przytycki-Sikora \cite{Bul:character}\cite{PS00}, the evaluation $A=-1$ yields a surjective map to the $\mathrm{SL}_2(\mathbb{C})$-character variety of $M.$ Nonetheless, a surprising conjecture of Witten states that the skein modules of closed $3$-manifolds with $\QQ(A)$-coefficients are finitely generated:
\begin{conjecture}\label{conj:closed}(Witten's skein module finiteness conjecture) Let $M$ be a closed compact oriented $3$-manifold. Then the Kauffman bracket skein module $K(M,\QQ(A))$ is a finite dimensional $\QQ(A)$-vector space.
\end{conjecture}
The first written mention of this conjecture can be found in \cite{Car17} and a detailed exposition in \cite{GM18}. Many skein modules of closed $3$-manifolds had already been computed (often with $\Z[A,A^{-1}]$ coefficients) and in all those cases, the module $K(M,\QQ(A))$ is indeed finite dimensional. We give a list of closed $3$-manifolds that are known to satisfy the conjecture: $S^3$ (where $K(S^3,\QQ(A))=\QQ(A)$; this is equivalent to the existence and unicity of the Kauffman bracket), $S^2\times S^1$ and lens spaces by Hoste-Przytycki \cite{HP93}\cite{HP95}, integer Dehn-filling on the trefoil knot by Bullock \cite{Bul:trefoil}, the quaternionic manifold by Gilmer and Harris \cite{GH07}, some Dehn-fillings of $(2,2b)-$torus links by Harris \cite{Har10}, a family of prime prism manifolds by Mroczkowski \cite{Mro11}, the $3$-torus by Carrega \cite{Car17}. Moreover Przytycki showed that for a connected sum of $3$-manifolds, $K(M_1\# M_2,\QQ(A))=K(M_1,\QQ(A))\otimes K(M_2,\QQ(A)),$ hence Conjecture \ref{conj:closed} is stable under connected sums. We note however that to the author's knowledge, there was no hyperbolic $3$-manifold $M$ for which it was known that $K(M,\QQ(A))$ is finite dimensional. One of the results of this article is to give examples of such $3$-manifolds:
\begin{theorem}\label{thm:examples} Let $K$ be a two-bridge knot or a torus knot, $E_K$ be the knot complement and $E_K(r)$ be the surgery on $K$ of slope $r.$ Then for all $r$ except at most finitely many, $K(E_K(r),\QQ(A))$ is finite dimensional, that is, $E_K(r)$ satisfies the finiteness conjecture.
\end{theorem}
As two-bridge knots, with the exception of $(2,2n+1)$-torus knots, are hyperbolic, and for an hyperbolic knot, all Dehn-fillings but at most finitely many are hyperbolic by Thurston's hyperbolic Dehn surgery theorem \cite{Th}, Theorem \ref{thm:examples} gives infinite families of closed compact oriented $3$-manifolds which satisfy Conjecture \ref{conj:closed}.

Theorem \ref{thm:examples} relies on the work of Le \cite{Le:2bridge} which computed the skein module of two-bridge knots and March\'e \cite{Mar10} for the skein module of torus knots. To state their results, first we note that if $M$ is a manifold with boundary, $K(M,\QQ(A))$ has a natural module structure over $K(\partial M \times [0,1],\QQ(A)).$ Then the skein module of the knot complement $E_K$ where $K$ is either a two-bridge knot or a torus knot is finitely generated over $\QQ(A)[m],$ where $m$ is the meridian of $K,$ viewed as an element of $K(\partial E_K,\QQ(A)).$

With that in mind, Theorem \ref{thm:examples} will be obtained as a direct corollary of our main result: 
\begin{theorem}\label{thm:Dehnfilling}
Let $K\subset S^3$ be a knot, $E_K=S^3\setminus K,$ $m \subset \partial E_K$ be the meridian of $K$ and assume that $K(E_K,\QQ(A))$ is finitely generated over $\QQ(A)[m].$ Then for all $r \in \QQ \cup \lbrace \infty \rbrace$ except possibly finitely many $r,$ $K(E_K(r),\QQ(A))$ is finite dimensional.
\end{theorem}
We will get the above theorem from the presentation of the skein algebra $K(T^2,\QQ(A))$ of the torus as the quantum torus $\mathcal{T}^{\theta}$ by Frohman-Gelca \cite{FG00} and defining "annihilating polynomials" of skein elements in $K(E_K,\QQ(A))$ (see Lemma \ref{lem:minpoly}). The finite set of slopes will appear as the set slopes of the Newton polygons of the annihilating polynomials of the generators of $K(E_K,\QQ(A))$ over $\QQ(A)[m].$

The paper is organized as follows: in Section \ref{sec:prelim} we recall a few well-known facts about skein modules (algebra structure for thickened surfaces, module structure for manifolds with boundary, the quantum torus isomorphism of Frohman-Gelca) and introduce some notations. In Section \ref{sec:conj} we discuss how to generalize the finiteness conjecture to manifolds with boundary, and provide a few examples. Section \ref{sec:mainthm} is where we prove Theorem \ref{thm:Dehnfilling}. Finally, in Section \ref{sec:comments} we give comments on the proof of Theorem \ref{thm:Dehnfilling} and how it could generalize in other settings.
\section{Preliminaries on skein modules}
\label{sec:prelim}
In this section, we recall some well-known properties of skein modules of $3$-manifolds: the module structure over the boundary skein module, and the connection between the skein algebra of a $2$-dimensional torus $T^2$ and the quantum torus. All results stated here would be valid with $\Z[A,A^{-1}]$ coefficients, but we will only care about skein modules with $\QQ(A)$ coefficients in the rest of the paper.

Let $\Sigma$ be a compact oriented surface. We recall two basic facts about the skein modules $K(\Sigma \times [0,1],\QQ(A))$ of thickened surfaces:
\begin{enumerate}
\item  The set of multicurves (disjoint union of simple closed curves) in $\Sigma \times \lbrace 1/2 \rbrace$ is a basis of the skein module $K(\Sigma \times [0,1],\QQ(A)).$\cite{HP92}
\item The skein module $K(\Sigma \times [0,1],\QQ(A))$ has a natural algebra structure given by the stacking product:
Given two multicurves $\alpha$ and $\beta$ in the surface $\Sigma$ their product is $\alpha \times \lbrace 2/3 \rbrace \cup \beta \times \lbrace 1/3 \rbrace \in K(\Sigma \times [0,1],\QQ(A)),$ that is, $\alpha \cdot \beta$ is the link obtained by stacking $\alpha$ on top of $\beta.$ For a general element of $K(\Sigma \times [0,1],\QQ(A))$ the product is extended by bilinearity, and gives $K(\Sigma \times [0,1],\QQ(A))$ an algebra structure. For $\Sigma$ a compact oriented surface, we will usually abbreviate the skein algebra $K(\Sigma \times [0,1],\QQ(A))$ as  $K(\Sigma,\QQ(A)).$
\end{enumerate}

Now consider $M$ a compact oriented manifold with boundary surface $\partial M.$ The boundary $\partial M$ has a neighborhood in $M$ that is homeomorphic to $\partial M \times [0,1].$ Moreover, there is an homeomorphism 
$$M\simeq M \underset{\partial M=\partial M \times \lbrace 0 \rbrace}{\coprod} \partial M \times [0,1].$$ This induces a $K(\partial M,\QQ(A))$ structure on $K(M,\QQ(A)),$ given by stacking:

If $L \subset M$ is a banded link in $M$ and $\alpha$ is a multicurve on $\partial M,$ then we define $\alpha \cdot L \in K(M,\QQ(A))$ as the link obtained by pushing $L$ inside $M$ and stacking the multicurve $\alpha$ on top of it in the $\partial M \times [0,1]$ component. Again, for general skein elements, the scalar product is extended by bilinearity, and it induces a $K(\partial M,\QQ(A))$-module structure on $K(M,\QQ(A)).$

The case of skein algebra $K(T^2,\QQ(A))$ of the $2$-dimensional torus $T^2$ has caught special attention because of the connection with the quantum torus algebra found independently by Sallenave \cite{Sal99} and Frohman-Gelca \cite{FG00}.

Let $\mathcal{T}$ be the non-commutative $\QQ(A)$-algebra:
$$\mathcal{T}=\QQ(A)\langle u,v \rangle/_{uv=A^2vu},$$
which is called the \textit{quantum torus} algebra. We define a $\QQ(A)$-basis of $\mathcal{T}$ as vector space by setting for any $(\alpha,\beta)\in \Z^2,$
$$e_{\alpha,\beta}=A^{-\alpha \beta}u^{\alpha}v^{\beta}.$$
We note that the product of basis elements is given by 
\begin{equation}\label{eq:prod1}e_{\alpha,\beta}\cdot e_{\mu,\nu}=A^{\alpha \nu -\beta\mu} e_{\alpha+\mu,\beta+\nu}.
\end{equation}
By the above formula, the $\QQ(A)$-linear map
$$\begin{array}{rccl}\theta :  & \mathcal{T} & \longrightarrow & \mathcal{T}
\\ &  e_{\alpha,\beta} & \longrightarrow &  e_{-\alpha,-\beta}\end{array}$$
is an involution of the algebra $\mathcal{T}.$ A basis, indexed by $(\alpha,\beta) \in \Z^2/_{\lbrace \pm 1 \rbrace }$ of the subalgebra $\mathcal{T}^{\theta}$ of $\theta$-invariant elements of $\mathcal{T}$ is given by:
$$\tilde{e}_{\alpha,\beta}=e_{\alpha,\beta}+e_{-\alpha,\beta}=A^{-\alpha\beta}(u^{\alpha}v^{\beta}+u^{-\alpha}v^{-\beta}).$$
Moreover, we have, for $(\alpha,\beta), (\mu,\nu) \in \Z^2/_{\lbrace \pm 1 \rbrace }:$
\begin{equation}\label{eq:prod2}\tilde{e}_{\alpha,\beta} \cdot \tilde{e}_{\mu,\nu}=A^{\alpha \nu -\beta\mu} \tilde{e}_{\alpha+\mu,\beta+\nu}+A^{\beta \mu -\alpha\nu} \tilde{e}_{\alpha-\mu,\beta-\nu}.
\end{equation}
On the other hand, recall that $K(T^2,\QQ(A))$ has the set of multicurves on $T^2$ as basis. Recall that any non trivial simple closed curve on $T^2$ is of the form $\gamma_{q/p}=l^pm^q$ where $l,m$ are the meridian and longitude, and $p,q$ are coprime integers.

For $n\geqslant 0,$ let $T_n(x)$ be the $n$-th Chebychev polynomial, defined by 
$$T_0(x)=2, \ T_1(x)=x \ \textrm{and} \ xT_n(x)=T_{n+1}(x)+T_{n-1}(x).$$
For any pair $(p,q) \in \Z^2/_{\lbrace \pm 1 \rbrace },$ Frohman-Gelca \cite{FG00} defines an element
$$(p,q)_T=T_{gcd(p,q)}(\gamma_{q/p}),$$
with the convention that $(0,0)_T=T_0(\emptyset)=2 \emptyset,$ where $\emptyset$ is the empty multicurve, which is the unit of $K(T^2,\QQ(A)).$ One then has that:
\begin{theorem}\cite{FG00}\label{thm:isom} The map
$$\begin{array}{rccl}\varphi: & K(T^2 ,\QQ(A)) & \longrightarrow & \mathcal{T}^{\theta}
\\ & (p,q)_T & \longrightarrow & \tilde{e}_{p,q} 

\end{array}$$
is an isomorphism of algebras.
\end{theorem}
For $M$ a manifold with toric boundary, the Frohman-Gelca isomorphism gives a nice way to understand the action of $K(T^2,\QQ(A))$ on $K(M,\QQ(A)).$ It can be used to relate the skein module of a knot complement to the $q$-holonomicity of the colored Jones polynomials \cite{Gel02} and to the AJ conjecture (see \cite{Le:2bridge} for example).

We will use this isomorphism to prove our Dehn-filling result in Section \ref{sec:mainthm}. 
\section{Witten's finiteness conjecture and generalizations}
\label{sec:conj}
The module structure on $K(M,\QQ(A))$ described in Section \ref{sec:prelim} suggests the following straightforward generalization of the finiteness conjecture:
\begin{conjecture}\label{conj:boundary1}(Finiteness conjecture for manifolds with boundary)

 Let $M$ be a compact oriented $3$-manifold. Then $K(M,\QQ(A))$ is a finitely generated $K(\partial M,\QQ(A))$-module.
\end{conjecture}
We note that Conjecture \ref{conj:closed} can be thought as a special case of the above conjecture, as the skein algebra of the empty surface is simply $\QQ(A).$

 However, we argue that one ought to expect a stronger statement than this one. To begin with, let us consider the case of skein modules of link complements. We already mentioned the case of two-bridge knots \cite{Le:2bridge} and torus knots \cite{Mar10}, where in both case, the skein module is finitely generated not only over $K(\partial E_K,\QQ(A))=K(T^2,\QQ(A))$ but over $\QQ(A)[m]$ where $m$ is the meridian. Moreover, the skein module of $E_L$ a two-bridge link complement has been computed by Le and Tran \cite{LT14}, and in that case $K(E_L,\QQ(A))$ is finitely generated over $\QQ(A)[m_1,m_2]$ where $m_1$ and $m_2$ are the meridians of the two components. Those examples lead us to propose the following conjecture in the case of link complements in $S^3:$
\begin{conjecture}\label{conj:links}(Finiteness conjecture for links)

 Let $L=\underset{i=1}{\overset{n}{\cup}} L_i$ be a $n$-th component link with $n\geqslant 1.$ Let $m_1,\ldots, m_n$ be the meridians of $L_1,\ldots,L_n.$ 
 
Then $K(E_L,\QQ(A))$ is a finitely generated $\QQ(A)[m_1,\ldots,m_n]$-module.
\end{conjecture}
We note that in the above listed examples, the module $K(E_L,\Z[A,A^{-1}])$ is even a free $\Z[A,A^{-1}][m_1,\ldots,m_n]$-module of finite rank. One should not expect this pattern to be general; indeed, torus knots and two-bridge knots and links complements are small, and it is thought that closed incompressible surfaces create torsion in the skein modules (see \cite{Oht02}, chapter 4.1 for a discussion of this phenomenon).

The pattern we saw for links hints that the skein module $K(M,\QQ(A))$ of a $3$-manifold with boundary may be finitely generated over a smaller algebra than the whole skein algebra $K(\partial M,\QQ(A))$ of the boundary. We propose the following:
\begin{conjecture}\label{conj:boundary2} (Strong finiteness conjecture for manifolds with boundary)

Let $M$ be a compact oriented $3$-manifold. Then there exists a finite collection $\Sigma_1,\ldots ,\Sigma_k$ of essential subsurfaces $\Sigma_i \subset \partial M$ such that:
\begin{itemize}
\item[-]For each $i,$ the dimension of $H_1(\Sigma_i,\QQ)$ is half that of $H_1(\partial M,\QQ).$ 
\item[-]The skein module $K(M,\QQ(A))$ is a sum of finitely many subspaces $F_1,\ldots,F_k,$ where $F_i$ is a finitely generated $K(\Sigma_i,\QQ(A))-$module.
\end{itemize}
\end{conjecture}
Note that we do not assume the sum of the subspaces $F_i$ to be direct. It is clear that this conjecture implies the weaker Conjecture \ref{conj:boundary1}.

Moreover if a link complement $E_L=S^3\setminus L$ satisfies Conjecture \ref{conj:links}, then the relevant collection of subsurfaces of $\partial E_L=\underset{i=1}{\overset{n}{\coprod}} T^2$ may be obtained by taking one surface $\Sigma,$ which is the disjoint union of one annulus neighborhood of each meridian $m_i$ in each component $L_i$ of the link. Moreover the collection of subspaces is reduced to $F_1=K(E_L,\QQ(A)).$

Compared with Conjecture \ref{conj:boundary1}, Conjecture \ref{conj:boundary2} morally states that the module $K(M,\QQ(A))$ has "half the dimension" of $K(\partial M,\QQ(A)).$

We motivate the above statement with the following examples:
\begin{proposition}\label{prop:examples}
We have:
\begin{enumerate}
\item For any $g\geqslant 2,$ the handlebody $H_g$ of genus $g$ satisfies Conjecture \ref{conj:boundary2}.
\item For any closed compact oriented surface $\Sigma,$ the thickened surface $\Sigma \times [0,1]$ satisfies Conjecture \ref{conj:boundary2}.
\end{enumerate}
\end{proposition}
\begin{figure}[!h]
\centering
\def \svgwidth{.6\columnwidth}
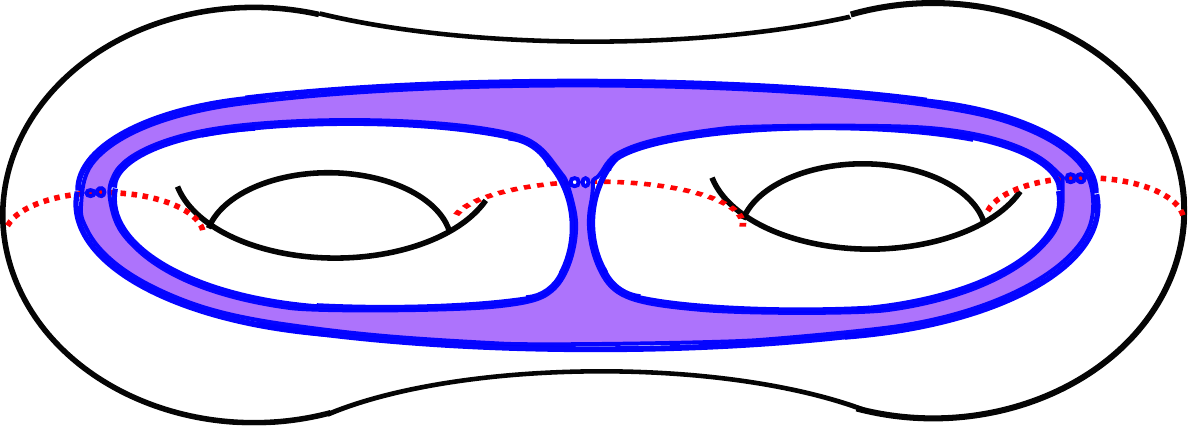
\caption{A handlebody $H_g$ of genus $g=2,$ in red a pair of pants decomposition by curves that bound disks in $H_g,$ in purple a dual graph $\Gamma.$ Notice that the pair $(H_g,\Gamma)$ is homeomorphic to $(\Gamma \times [0,1],\Gamma \times \lbrace 1 \rbrace).$}
\label{fig:handlebody}
\end{figure}
\begin{proof}
In both cases, the collection of subsurfaces will consist of one subsurface only, and the collection of subspaces will just be $F_1=K(M,\QQ(A)).$

For the first claim, pick a pair of pants decomposition $\mathcal{C}$ of $\Sigma_g =\partial H_g$ by simple closed curves which bound disks in $H_g,$ and let $\Gamma$ be a dual trivalent banded graph to the decomposition: $\Gamma$ has one trivalent vertex in each pair of pants, each of its edges intersect exactly one curve of the pair of pants decomposition in exactly one segment. See Figure \ref{fig:handlebody}.
Then the pair $(H_g,\Gamma)$ is homeomorphic to $(\Gamma \times [0,1], \Gamma \times \lbrace 1 \rbrace,$ and thus $K(H_g,\QQ(A))$ is generated over $K(\Gamma,\QQ(A))$ by the empty link.

For the second claim, we note that the boundary of $\Sigma \times [0,1]$ is $\Sigma \times \lbrace 0 \rbrace \coprod \Sigma \times \lbrace 1 \rbrace.$ It is clear that $K(\Sigma \times [0,1],\QQ(A))$ is generated over $K(\Sigma \times \lbrace 1 \rbrace,\QQ(A))$ by the empty link.
\end{proof}
Although, the above examples required the use of one subsurface only, in general you need at least a collection of subsurfaces $(\Sigma_i)_{1\leqslant i \leqslant k}$ and subspaces $(F_i)_{1\leqslant i \leqslant k}.$ 

Let us illustrate this using the computation by Dabkowski and Mroczkowski \cite{DM09} of $K(\Sigma_{0,3}\times S^1,\Z[A,A^{-1}])$ where $\Sigma_{0,3}$ is the pair of pants. Let us recall their result:

\begin{figure}[!h]
\centering
\def \svgwidth{.8\columnwidth}
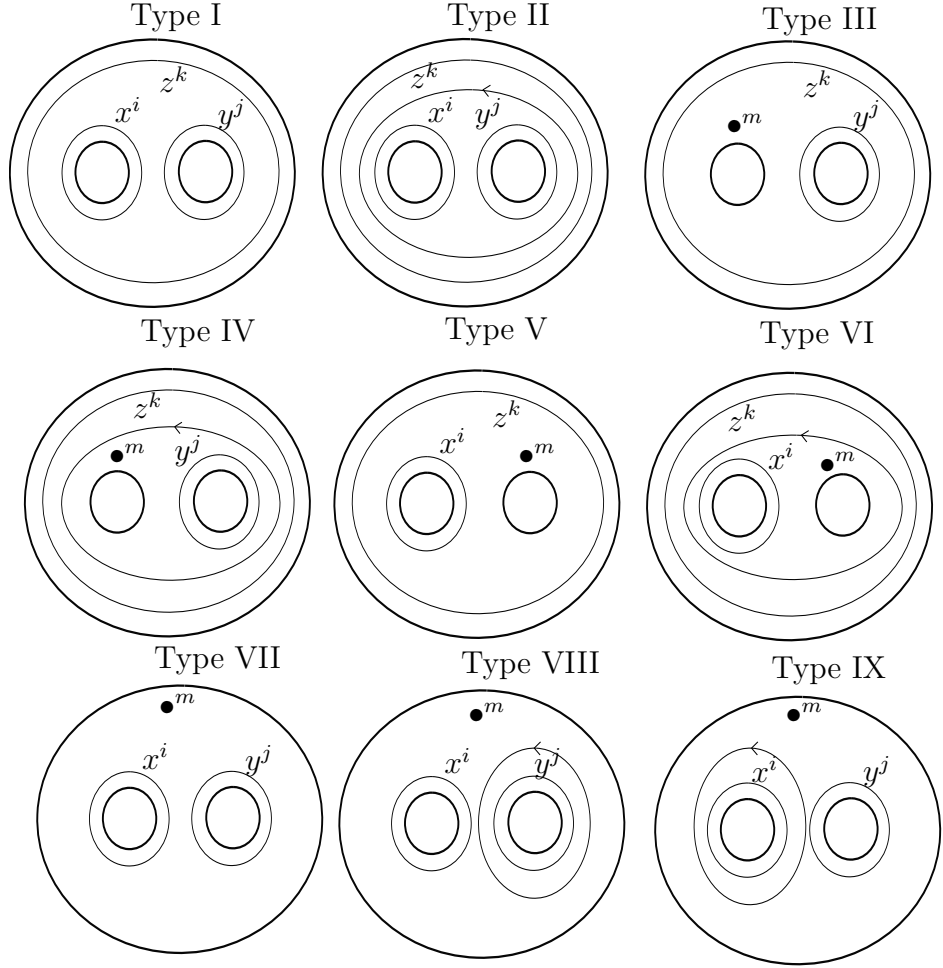
\caption{The nine types of generators of $K(\Sigma_{0,3}\times S^1,\Z[A,A^{-1}]).$ The curves $x,y$ and $z$ are horizontal curves and circle around each hole in $\Sigma_{0,3},$ and $x^i$ denotes $i$ parallel copies of $x.$ The curves with arrows are curves which travel up along the $S^1$ factor near the arrow. The dots $\bullet$ denote vertical curves of the form $pt \times S^1$ and $\bullet^m$ means there are $m$ dots. The same generators may appear in different families; moreover, to get a basis one needs to assume $j=0$ in the last family.}
\label{fig:disk2holes}
\end{figure}
\begin{theorem}\cite{DM09} \label{thm:disk2holes} Let $\Sigma_{0,3}$ be the sphere with three holes. Then $K(\Sigma_{0,3}\times S^1,\Z[A,A^{-1}])$ is free over the nine types of generators described in Figure \ref{fig:disk2holes}
\end{theorem}
We verify the consistency of Conjecture \ref{conj:boundary2} with the above theorem:
\begin{proposition}\label{prop:disk2holes} $\Sigma_{0,3} \times S^1$ satisfies Conjecture \ref{conj:boundary2}.
\end{proposition}
\begin{proof}
Let $F_1$ be the subspace of $K(\Sigma_{0,3},\QQ(A))$ generated by the generators of type I and II, $F_2$ the subspace generated by those of type III and IV, $F_3$ the one generated by type V and VI, and finally, $F_4$ that obtained from type VII, VIII and IX generators. Let $T_1,T_2,T_3$ be the three torus boundary components, corresponding to the curves $x,y$ and $z$ respectively. Then let $\Sigma_1$ be the union of three horizontal annuli, one in each $T_i,$ and $\Sigma_2$ (resp. $\Sigma_3,\Sigma_4$) be the subsurfaces that are union of one vertical annulus in $T_1$ (resp. $T_2$,$T_3$) and one horizontal annulus in each of the other two torus boundaries.
Then one can see that $F_1$ (resp. $F_2,$ $F_3,$ and $F_4$) is finitely generated over $K(\Sigma_i,\QQ(A))$ as multiplication by the middle curve in a horizontal annulus in $T_1$ (resp. $T_2,$ $T_3$) add one component of $x$ (resp. $y,$ $z$) and multiplication by the middle curve in a vertical annulus adds one dot. 

If $\hat{x},\hat{y}$ and $\hat{z}$ denote the curves $x,y$ and $z$ with an arrow, a set of generators for $F_1$ (resp. $F_2,$ $F_3$ and $F_4$) is $(\emptyset, \hat{z})$ (resp. $(\emptyset, \hat{z})$, $(\emptyset,\hat{z})$ and $(\emptyset,\hat{x},\hat{y})$).
\end{proof}

\section{A stability theorem under generic Dehn-filling of knots} 
\label{sec:mainthm}
We now turn to the proof of Theorem \ref{thm:Dehnfilling}. In this section $K$ will be a knot in $S^3$ with meridian $m$ and $E_K=S^3\setminus K$ will be the knot complement.

We first start with a lemma explaining that under the hypothesis that $K(E_K,\QQ(A))$ is finitely generated over $\QQ(A)[m],$ any element $f\in K(E_K,\QQ(A))$ has an "annihilating polynomial" $P\neq 0 \in K(\partial E_K,\QQ(A))=K(T^2,\QQ(A))$ such that $P\cdot f=0 \in K(E_K,\QQ(A)).$ 

We express this idea in the quantum torus formulation of $K(T^2,\QQ(A)):$
\begin{lemma}\label{lem:minpoly}(Annihilating polynomials)

 Let $K$ be a knot with meridian $m$ such that $K(E_K,\QQ(A))$ is finitely generated over $\QQ(A)[m].$ Let $f \in K(E_K,\QQ(A)).$ Then there is a polygon $\mathcal{P}\subset \RR^2$ with vertices in $\Z^2$ and coefficients $c_{\alpha,\beta}(A)\in \QQ(A)$ for every $(\alpha,\beta)\in \Z^2 \cap \mathcal{P},$ such that:
\begin{itemize}
\item[-]$(-\mathcal{P})=\mathcal{P}$ and for every $(\alpha,\beta)\in \Z^2\cap \mathcal{P},$ we have $c_{\alpha,\beta}(A)=c_{-\alpha,-\beta}(A).$
\item[-]If $(\alpha,\beta)$ is a vertex of $\mathcal{P}$ then $c_{\alpha,\beta}(A)\neq 0.$
\item[-]We have the equality: \begin{equation}\label{eq:minpol}\left(\underset{(\alpha,\beta) \in \mathcal{P}\cap \Z^2}{\sum} c_{\alpha,\beta}(A) e_{\alpha,\beta} \right)\cdot f =0
\end{equation}
 in $K(E_K,\QQ(A)).$
\end{itemize}
\end{lemma}
\begin{proof}
Let $l$ be the longitude of $K.$ The infinite family $(l^i\cdot f)_{i \in \BN}$ is an infinite family of element of $K(E_K,\QQ(A))$ which is a finitely generated $\QQ(A)[m]$-module. Thus there exists a non trivial linear dependence relation:
$$a_d(m)l^d\cdot f + \ldots +a_0(m) \cdot f=0,$$
where $d$ is an integer and the $a_i(m) \in \QQ(A)[m]$ do not all vanish. 

One can restate this saying there is a non-zero element 
$$Q=\underset{i=0}{\overset{d}{\sum}} a_i(m)l^i \in K(\partial E_K,\QQ(A))$$ such that $Q\cdot f=0.$ Using the isomorphism \ref{thm:isom}, we get a relation:
$$\left(\underset{(\alpha,\beta) \in \Z^2}{\sum}c_{\alpha,\beta}(A) e_{\alpha,\beta} \right) \cdot f=0$$
for some coefficients $c_{\alpha,\beta}(A)\in \QQ(A),$ where the sum is a finite sum that is invariant under $\theta.$ Let $\mathcal{P}$ be the convex hull of all $(\alpha,\beta)$ such that $c_{\alpha,\beta}(A)\neq 0.$ Then, as the sum is invariant under $\theta,$ we have that $(-\mathcal{P})=\mathcal{P}$ and $c_{-\alpha,-\beta}(A)=c_{\alpha,\beta}(A)$ for every $(\alpha,\beta) \in \Z^2.$ Finally,
$$\left(\underset{(\alpha,\beta) \in \mathcal{P}\cap \Z^2}{\sum} c_{\alpha,\beta}(A)e_{\alpha,\beta}\right) \cdot f=0$$
is the required relation.
\end{proof}
Because of the module structure on $K(E_K,\QQ(A)),$ for any $f \in K(E_K,\QQ(A))$ the set of elements $Q \in K(\partial E_K,\QQ(A))$ such that $Q\cdot f=0$ is a left ideal of $K(E_K,\QQ(A)).$ Hence, once has an annihilating polynomial for $f$ as in Equation \ref{eq:minpol}, one gets many other relations:
\begin{lemma}\label{lem:minpoltranslated}Let $f \in K(E_K,\QQ(A))$ that satisfies Equation \ref{eq:minpol}.

 Then, for any $(\mu,\nu) \in \Z^2,$ we have:
\begin{equation}
\label{eq:minpoltranslated}\left(\underset{(\alpha,\beta) \in \mathcal{P} \cap \Z^2}{\sum} c_{\alpha,\beta}(A)A^{\beta \mu -\alpha \nu} \tilde{e}_{\alpha+\mu,\beta+\nu}\right) \cdot f=0.
\end{equation}
\end{lemma}
\begin{proof}
We multiply the annihilating polynomial of $f$ given in Equation  \ref{eq:minpol} on the left by $\tilde{e}_{\mu,\nu}=e_{\mu,\nu}+e_{-\mu,-\nu} \in \mathcal{T}^{\theta}.$ We get:
\begin{multline*} (e_{\mu,\nu}+e_{-\mu,-\nu})\cdot \left(\underset{(\alpha,\beta) \in \mathcal{P}\cap \Z^2}{\sum} c_{\alpha,\beta}(A)e_{\alpha,\beta}\right)
\\ = \underset{(\alpha,\beta)\in \mathcal{P}\cap \Z^2}{\sum} c_{\alpha,\beta}(A)A^{\beta\mu -\alpha\nu} e_{\alpha+\mu,\beta+\nu} + \underset{(\alpha,\beta)\in \mathcal{P}\cap \Z^2}{\sum} c_{\alpha,\beta}(A)A^{\alpha \nu -\beta \mu} e_{\alpha-\mu,\beta-\nu}
\\ = \underset{(\alpha,\beta)\in \mathcal{P}\cap \Z^2}{\sum} c_{\alpha,\beta}(A)A^{\beta\mu -\alpha\nu} e_{\alpha+\mu,\beta+\nu} + \underset{(-\alpha,-\beta)\in \mathcal{P}\cap \Z^2}{\sum} c_{\alpha,\beta}(A)A^{\beta\mu -\alpha\nu} e_{-\alpha-\mu,-\beta-\nu}
\\ = \underset{(\alpha,\beta)\in \mathcal{P}\cap \Z^2}{\sum} c_{\alpha,\beta}(A)A^{\beta\mu -\alpha\nu} \tilde{e}_{\alpha+\mu,\beta+\nu}
\end{multline*}
where the second equality uses the fact that $(-\mathcal{P})=\mathcal{P}$ and $c_{-\alpha,-\beta}(A)=c_{\alpha,\beta}(A).$ The lemma then follows from the left ideal structure of the set 
$$\lbrace Q \in K(\partial E_K,\QQ(A)) \ \textrm{such that} \ Q\cdot f=0 \rbrace.$$
\end{proof}
The relations given by Equation \ref{eq:minpoltranslated} allow us to say that the submodule $K(\partial E_K,\QQ(A))\cdot f$ is generated by the elements $\tilde{e}_{(\alpha,\beta)}$ where $(\alpha,\beta)$ belongs in a band of finite width in $\RR^2:$
\begin{lemma}\label{lem:band} Let $f \in K(E_K,\QQ(A))$ with annihilating polynomial as in Equation \ref{eq:minpol}. Let $\lambda :\Z^2 \rightarrow \Z$ be a homomorphism such that the maximum $M$ of $\lambda$ on $\mathcal{P}$ is attained in a unique point of $\mathcal{P}.$

Then $K(\partial E_K,\QQ(A))$ is spanned by the family 
$$\lbrace \tilde{e}_{\alpha,\beta} \cdot f \ \textrm{for} \ (\alpha,\beta) \in \Z^2/_{\lbrace \pm 1 \rbrace } \ \textrm{such that} \ |\lambda(\alpha,\beta)|\leqslant M \rbrace.$$
\end{lemma}
\begin{proof}
We need to show that if $|\lambda (\alpha,\beta)| \geqslant M,$ then $\tilde{e}_{\alpha,\beta} \cdot f$ is a linear combination of elements $\tilde{e}_{\alpha',\beta'}\cdot f,$ with $|\lambda ( \alpha',\beta')|<|\lambda(\alpha,\beta)|.$ As $\tilde{e}_{-\alpha,-\beta}=\tilde{e}_{\alpha,\beta}$ by definition, let us assume that $\lambda(\alpha,\beta) \geqslant 0.$ Let $(x,y)$ be the vertex of $\mathcal{P}$ such that $\lambda(x,y)=M.$

We can choose $(\mu,\nu)$ in Equation \ref{eq:minpoltranslated} so that $(x+\mu,y+\nu)=(\alpha,\beta).$ 
As the vertex coefficients are non-zero, and as $\mathcal{P}$ is symmetric, Equation \ref{eq:minpoltranslated} allows one to express $\tilde{e}_{\alpha,\beta}\cdot f$ in terms elements $\tilde{e}_{\alpha'\beta'}\cdot f$ with $\lambda(\alpha,\beta)-2M \leqslant \lambda(\alpha',\beta') <\lambda(\alpha,\beta).$
\end{proof}
Now having a good understanding of the skein module $K(E_K,\QQ(A)),$ we turn to the effect of the Dehn-filling. 
Let us list some of the new relations created by the Dehn-filling of slope $q/p:$
\begin{lemma}\label{lem:dehn_fil_rel}Let $(p,q)$ be coprime integers, and $E_K(q/p)$ be the Dehn-filling of the knot complement $E_K$ of slope $q/p\in \QQ \cup \lbrace \infty \rbrace.$ Then the map 
$$i_*:K(E_K,\QQ(A)) \longrightarrow K(E_K(q/p),\QQ(A))$$ induced by the inclusion $i:E_K \rightarrow E_K(q/p)$ is surjective, and moreover, 

for any $(\alpha,\beta) \in \Z^2$ and any $f \in K(E_K,\QQ(A)),$ one has
$$ \left(A^{p\beta-q\alpha}\tilde{e}_{\alpha+p,\beta+q} +A^{q\alpha-p\beta}\tilde{e}_{\alpha-p,\beta-q}+(A^2+A^{-2})\tilde{e}_{\alpha,\beta} \right)\cdot f=0$$
in $K(E_K(q/p),\QQ(A)).$
\end{lemma}
\begin{proof}
First, we note that any link in $E_K(q/p)$ can be isotoped to be disjoint from the solid torus of the Dehn-filling. Hence $i_*$ is surjective.

Then, notice that the peripheral curve $\gamma_{q/p}$ of slope $q/p$ bounds a disk in $E_K(q/p).$ By the Kauffman relation K1, if $f \in K(E_K,\QQ(A))$ then 
$$\gamma_{q/p} \cdot f = (-A^2-A^{-2})f$$ in $K(E_K(q/p),\QQ(A)).$ 

Using the isomorphism of Theorem \ref{thm:isom}, this is equivalent to 
$$(\tilde{e}_{p,q}+(A^2+A^{-2})e_{0,0})\cdot f=0 \in K(E_K(q/p),\QQ(A)).$$
Notice that this relation is true for any $f\in K(E_K,\QQ(A)).$ We apply it to $\tilde{e}_{\alpha,\beta}\cdot f$ instead of $f$ and get:
\begin{multline*}(\tilde{e}_{p,q}+(A^2+A^{-2})e_{0,0})\cdot (\tilde{e}_{\alpha,\beta}\cdot f)
\\ =\left(A^{p\beta-q\alpha}\tilde{e}_{\alpha+p,\beta+q}+A^{q\alpha-p\beta}\tilde{e}_{\alpha-p,\beta-q}+(A^2+A^{-2})\tilde{e}_{\alpha,\beta}\right) \cdot f=0 \in K(E_K(q/p),\QQ(A)).
\end{multline*}
\end{proof}
Note that the relations listed above do not necessarily generate the kernel of the map $K(E_K,\QQ(A))\rightarrow K(E_K(q/p),\QQ(A)).$ Actually, whenever a $3$-manifold $M_2$ is obtained by attaching a $2$-handle to a $3$-manifold $M_1,$ the work of Hoste and Przytycki \cite{HP93} describes a set of generators for the kernel of the inclusion map 
$$i_*: K(M_1,\QQ(A)) \rightarrow K(M_2,\QQ(A)).$$
 The set of generators are elements of the form $L-L'$ where $L$ is a link in $M_1$ and $L'$ is obtained from $L$ by sliding one the components of $L$ along the $2$-handle.
 
In the above, we only used sliding of trivial components. However, this will turn out to be sufficient to prove our main theorem:
\begin{proof}[Proof of Theorem \ref{thm:Dehnfilling}]
Let $f_1,\ldots,f_d$ be a finite set of generators of $K(E_K,\QQ(A))$ over $\QQ(A)[m].$ Let $q/p\in \QQ \cup \lbrace \infty \rbrace$ be any slope. As $K(E_K(q/p),\QQ(A))$ is spanned by $K(E_K,\QQ(A)),$ to show that $K(E_K(q/p),\QQ(A))$ is finite dimensional, it is sufficient to show that for each $i,$ $K(\partial E_K,\QQ(A))\cdot f_i \subset K(E_K(q/p),\QQ(A))$ is finite dimensional.

\begin{figure}[!h]
\centering
\def \svgwidth{.45\columnwidth}
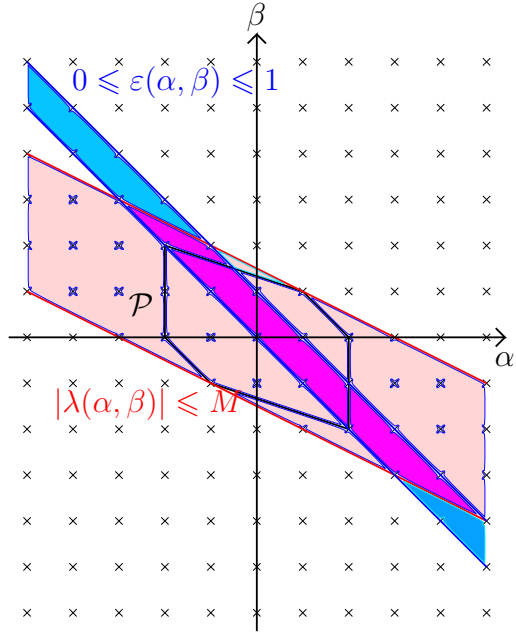
\caption{The polygon $\mathcal{P}$ has slopes $-1,-1/3$ and $\infty,$ and does not have $-1/2$ as a slope. The figure shows in red the band of direction $(p,q)=(2,-1),$ slope $-1/2$ and equation $|\lambda(\alpha,\beta)|=|\alpha+2 \beta|\leqslant 3$ and in blue the band $0\leqslant \varepsilon(\alpha,\beta)=\alpha+\beta\leqslant 1.$   In that case, the subspace $K(\partial E_K(-1/2),\QQ(A))\cdot f$ would be finitely dimensional, and, accounting for the $\pm 1$ symmetry, have dimension at most $11.$}
\label{fig:bands}
\end{figure}

Let $\mathcal{P}_1,\ldots, \mathcal{P}_d$ be the Newton polygons of annihilating polynomials of $f_1,\ldots ,f_d$ in $K(E_K,\QQ(A)),$ as introduced in Lemma \ref{lem:minpoly}. Assume that $q/p$ is not among the slopes of sides of the polygons $\mathcal{P}_1,\ldots,\mathcal{P}_d.$

(Note that some of those polygons might be degenerate; if $\mathcal{P}$ is a segment, we define its set of slopes to be the slope of that segment, if $\mathcal{P}$ is a point, we define its set of slopes to be empty.)

Fix $1\leqslant i \leqslant d$ and let us prove that $K(\partial E_K,\QQ(A))\cdot f_i \subset K(E_K(q/p),\QQ(A))$ is finite dimensional. Let $\lambda :\Z^2 \rightarrow \Z$ be a non-zero homomorphism such that $\lambda(p,q)=0,$ and let $\varepsilon : \Z^2 \rightarrow \Z$ be a homomorphism such that $\varepsilon(p,q)=1.$ Such a homomorphism $\varepsilon$ exists as $(p,q)$ are coprime.

As $q/p$ is not a slope of $\mathcal{P}_i,$ the homomorphism $\lambda$ has a unique maximum $M_i$ on $\mathcal{P}_i.$ By Lemma \ref{lem:band}, the subspace $K(\partial E_K,\QQ(A))\cdot f_i$ of $K(E_K(q/p),\QQ(A))$ is spanned by elements $\tilde{e}_{\alpha,\beta}\cdot f_i$ such that $|\lambda(\alpha,\beta)|\leqslant M_i.$

But using Lemma \ref{lem:dehn_fil_rel} repeatedly, any element $\tilde{e}_{\alpha,\beta}\cdot f_i$ may be expressed as a linear combination of elements $\tilde{e}_{\alpha',\beta'}\cdot f_i$ with $\lambda(\alpha',\beta')=\lambda(\alpha,\beta)$ and $0\leqslant \varepsilon(\alpha',\beta')\leqslant 1.$

So $K(\partial E_K,\QQ(A))\cdot f_i$ is spanned by elements of the form $\tilde{e}_{\alpha',\beta'} \cdot f_i$ where $(\alpha',\beta')$ is a lattice point in the intersection of two non-parallel bands inside $\RR^2.$ This is a finite set, see Figure \ref{fig:bands}.
\end{proof}
\section{Comments and possible extensions}
\label{sec:comments}
We conclude this article with a few comments on the proof of Theorem \ref{thm:Dehnfilling} and its possible extensions.
\\
\\
\\
1) We note that the proof of Theorem \ref{thm:Dehnfilling} used in a fundamental way that the skein module $K(E_K,\QQ(A))$ is not only finitely generated on $K(\partial E_K,\QQ(A))$ but finitely generated on $\QQ(A)[m].$ This leads us to believe that the correct generalization of the finiteness conjecture to manifolds with boundary is Conjecture \ref{conj:boundary2} rather than Conjecture \ref{conj:boundary1}, if one wants to have nice stability properties under Dehn-filling.
\\
\\
2) One may straighten Theorem \ref{thm:Dehnfilling} by working with any manifold $M$ with boundary $\partial M=T^2$ instead of a knot complement in $S^3.$ One would get the following:
\begin{corollary}Let $M$ be a $3$-manifold with boundary $\partial M=T^2$ that satisfies Conjecture \ref{conj:boundary2}. Then all Dehn-fillings of $M$ except at most finitely many satisfies the finiteness conjecture.
\end{corollary}
\begin{proof}
Each subsurface $\Sigma_i \subset T^2$ has to be an annulus neighborhood of some non-trivial simple closed curve $m_i \subset T^2.$ Then to show that the image of each subspace $F_i \subset K(M,\QQ(A))$ in $K(M(r),\QQ(A))$ is finite dimensional, the proof is essentially the same as the proof of Theorem \ref{thm:Dehnfilling}, with $m$ playing the role of the meridian.
\end{proof}
Note however that all known (irreducible) examples of manifolds $M$ with $\partial M=T^2$ that satisfy Conjecture \ref{conj:boundary2} are knot complements in $S^3.$
\\
\\
3) For link complements in $S^3$ our argument would sadly fail to show that Conjecture \ref{conj:boundary2} is stable under generic Dehn-filling of one torus boundary component. Indeed, assume for example that $E_L=S^3\setminus L$ where $L=L_1 \cup L_2$ is a two-component link with meridians $m_1,m_2,$ and let $E_L(q/p)$ be the Dehn-filling of slope $q/p$ along the component $L_1.$ If $\gamma \in K(E_L,\QQ(A)),$ looking at linear dependence relations between the infinite family $(l_1^i \cdot \gamma)_{i\in \BN},$ one would get an annihilating polynomial $P(m_1,l_1)$ for $\gamma$ with coefficients not in $\QQ(A)$ but in $\QQ(A)[m_2],$ which is no longer a field. One would however be able to show that the \textit{localized} skein module $$K(E_L(r),\QQ(A))_{m_2}=K(E_L(r),\QQ(A))\underset{\QQ(A)[m_2]}{\otimes} \QQ(A)(m_2)$$ is finitely generated: 
\begin{corollary}Let $L$ be a two component link in $S^3$ such that $K(E_L,\QQ(A))$ is finitely generated over $\QQ(A)[m_1,m_2].$ 

Let $E_L(r)$ be the Dehn-filling of $E_L$ of slope $r$ along the component $L_1.$ 

Then for almost all $r,$ the localized skein module $K(E_L(r),\QQ(A))_{m_2}$ is finitely generated over $\QQ(A)(m_2).$
\end{corollary}
\begin{proof}Again, the proof is essentially the same as the proof of Theorem \ref{thm:Dehnfilling}, except that in Lemma \ref{lem:minpoly}, the coefficients $c_{\alpha,\beta}$ are elements of $\QQ(A)(m_2)$ instead of $\QQ(A).$ As $\QQ(A)(m_2)$ is a field, the elimination arguments in the proof of Lemma \ref{lem:band} and Theorem \ref{thm:Dehnfilling} still apply.
\end{proof}
Localized skein modules of knots are relevant in the study of the AJ conjecture, see \cite{Le:2bridge}, \cite{LT15} and \cite{LZ:AJ}. We recall that the skein modules of two-bridge links are known to be finitely generated over the two meridians \cite{LT14}.
\\
\\
4) One may wonder if Theorem \ref{thm:Dehnfilling} applies for $\Z[A,A^{-1}]$-coefficients. Under an extra hypothesis on the annihilating polynomials of the generators, we can extend the theorem to $\Z[A,A^{-1}]$ coefficients:
\begin{corollary}\label{cor:Z[A]coeff}Let $K$ be a knot such that $K(E_K,\Z[A,A^{-1}])$ is finitely generated over $\Z[A,A^{-1}][m]$ with generators $f_1,\ldots,f_p.$ 

Assume that $f_1,\ldots,f_p$ have annihilating polynomials as in Lemma \ref{lem:minpoly} such that all their vertex coefficients are monomials $\pm A^k.$ 

Then for all slopes $r\in \QQ \cup \lbrace \infty \rbrace$ except at most finitely many, $K(E_K(r),\Z[A,A^{-1}])$ is finitely generated over $\Z[A,A^{-1}].$
\end{corollary}
\begin{proof}
We assume $\left( \underset{(\alpha,\beta) \in \mathcal{P}_i}{\sum} c_{\alpha,\beta}^i(A) e_{\alpha,\beta} \right)\cdot f_i=0$ in $K(E_K,\Z[A,A^{-1}])$ and let $r=p/q$ be any slope that is not a slope of some $\mathcal{P}_i.$

 Because the coefficients $c_{\alpha,\beta}^i(A)$ are invertible whenever $(\alpha,\beta)$ is a vertex of $\mathcal{P}_i,$ so are the vertex coefficients of the translated annihilating polynomials by Equation \ref{eq:minpoltranslated}. 
 
 Thus the proof of Lemma \ref{lem:band} works also in this setting and $K(E_K,\Z[A,A^{-1}])$ is generated by the $\tilde{e}_{\alpha,\beta}\cdot f_j$ where $(\alpha,\beta)$ belong in some band $|\lambda (\alpha,\beta)|\leqslant M$ where $\lambda :\Z^2\rightarrow \Z$ is a non-zero homomorphism such that $\lambda(p,q)=0.$

As the vertex coefficients in Lemma \ref{lem:dehn_fil_rel} are also invertible, the proof of Theorem \ref{thm:Dehnfilling} is also still applicable and shows that $K(E_K(r),\Z[A,A^{-1}])$ is generated over $\Z[A,A^{-1}]$ by the $\tilde{e}_{\alpha,\beta}\cdot f_j$ where $(\alpha,\beta)$ are lattice points in the intersection of two non-parallel bands.   
\end{proof}
If $K$ is a knot, the minimal annihilating polynomial of the empty link $\emptyset$ is related to the non-commutative $Â$-polynomial of the knot, which by the AJ conjecture should be a quantization of its $A$-polynomial. The $A$-polynomial $A_K(L,M) \in \Z[L^{\pm 1},M^{\pm 1}]$ is known to have $\pm 1$ as vertex coefficients\cite{Coo97}. Moreover, in all cases where the $\hat{A}$ polynomial has been computed, its vertex coefficients are monomials.
\\
\\
5) We note that the set of generators described in the proof of Theorem \ref{thm:Dehnfilling} has no reason of being a basis, as we did not even use the complete set of slide relations to prove finite dimensionality. 

There are two ways of studying linear independence in $K(M,\QQ(A))$ where $M$ is a $3$-manifold. A first way is to keep track of all the slides relations (described in \cite{HP93}) coming from the decomposition of $M$ into $0-,$ $1-$ and $2-$handles. This is the method used for all the early computations of skein modules listed in the introduction, but it can get intractable very fast. 

An alternative way is to use invariants. Firstly, the skein module $K(M,\QQ(A))$ carries a natural $H_1(M,\Z/2)$-grading as the Kauffman relations are homogeneous. Secondly, one can use the Gilmer-Masbaum map which sends an element in $K(M,\QQ(A))$ to the sequence of Reshetikhin-Turaev invariants associated to it. 

Those techniques have been carried out in \cite{Gil18} for the $3$-torus and partially  in the case of $M=\Sigma \times S^1$ where $\Sigma$ is a closed surface of genus $g\geqslant 2$ in \cite{GM18}. 

Gilmer and Masbaum asked the question whether the $H_1(M,\Z/2)$-grading and the Gilmer-Masbaum map distinguish all skeins in $K(M,\QQ(A)).$
\\
\\
6) In the case where $K$ is a two-bridge or torus knot, all our computations can be made more explicit. We have:
\begin{corollary}Let $K$ be a two-bridge or torus knot. Then one can algorithmically compute the finite set of slopes in Theorem \ref{thm:Dehnfilling}, and compute an explicit bound for the dimension of $K(E_K(r),\QQ(A))$ when $r$ is not in this set.

One can also algorithmically check if the condition of Corollary \ref{cor:Z[A]coeff} is satisfied.
\end{corollary}
\begin{proof} Indeed, to do all this one only needs to compute the annihilating polynomials of the generators $f_1,\ldots f_p$ of $K(E_K,\Z[A,A^{-1}])).$ But both the work of Le \cite{Le:2bridge} and the work of March\'e \cite{Mar10} give a method to algorithmically reduce skein elements of $K(E_K,\Z[A,A^{-1}])$ as a $\Z[A,A^{-1}][m]$ linear combination of the generators. Then one simply has to find a $\Z[A,A^{-1}][m]$ linear dependence between the $(l^i f_j)_{i \in \BN}$ for each $f_j$ to get the annihilating polynomials of the generators $f_1,\ldots ,f_p.$ The dimension bound comes from a count of lattice points in the intersection of two non-parallel bands as in Figure \ref{fig:bands}.
\end{proof}
7) Theorem \ref{thm:Dehnfilling} suggests a (naive) approach to proving Conjecture \ref{conj:closed} for "generic" $3$-manifolds. One may start with a handlebody $H_g,$ which, as stated in Proposition \ref{prop:examples}, satisfies Conjecture \ref{conj:boundary2}. Then, one would want to choose a "generic" Heegaard decomposition, attaching disks along some curves in $\partial H_g=\Sigma_g,$ so that at each step, the $3$-manifold obtained satisfies Conjecture \ref{conj:boundary2}. To do this we would need to prove a version of Theorem \ref{thm:Dehnfilling} for generic $2$-handles instead of generic Dehn-filling of knots.
\\
\\
8) Finally, for manifolds with boundary, relative versions of skein modules have been introduced, spanned not only by links but by arcs and links. If $M$ is a $3$-manifold with a set $\mathcal{S}$ of disjoint closed intervals embedded in $\partial M,$ the relative skein module of $(M,\mathcal{S})$ is
$$K(M,\mathcal{S})=\mathrm{Span}_{\QQ(A)}\lbrace \textrm{banded tangle T in M with} \ \partial T=\mathcal{S}\rbrace/_{\textrm{rel isotopy, K1,K2}}.$$
Relative skein modules are thought to be a more natural object with respects to cutting along surfaces, or to triangulations. Those ideas proved particularly fruitful for the study of skein algebras of surfaces \cite{BW16}\cite{FKL19}\cite{Le:triangular}.

One may want to formulate a finiteness conjecture also in that context. 
\subsection*{Acknowledgements} 
The author would like to thank L\'eo Benard, Anthony Conway, Stavros Garoufalidis, Filip Misev and Gauthier Ponsinet for helpful conversations. This work was realized while the author was host at the Max Planck Institute for Mathematics in Bonn. The author would like to thank the institute for its hospitality.

%%%%%%%%%%%%%%%%%%%%%%%%%%%%%%%%%%%%%%%%%%%%%%%%%%%%%%%%%%%%%%%%%%%%
%%%%%%%%%%%%%%%%%%%%%%%%%%%%%%%%%%%%%%%%%%%%%%%%%%%%%%%%%%%%%%%%%%%%

%\bibliographystyle{hamsplain}
\bibliographystyle{hamsalpha}
\bibliography{biblio}
\end{document}